\documentclass[12pt]{amsart}
\usepackage[utf8]{inputenc}
\usepackage{tikz} 
\usepackage{tikz-cd}
\usepackage[all]{xy}
\usepackage{amsthm} 
\usepackage{amsmath}
\usepackage{amssymb}
\usepackage{mathtools}
\usepackage{hyperref}
\usepackage[capitalise,noabbrev]{cleveref} 
\usepackage{geometry}
\usepackage{xcolor}
\usepackage{comment} 
\usepackage{todonotes}
\usepackage{graphicx}
\usepackage{enumitem}
\usepackage{lipsum}

\allowdisplaybreaks

\title{Cohomology Theories of Partial Groups}
\author{Sandro Pfammatter}
\address{Master’s program in Mathematics, Ecole Polytechnique F\'ed\'erale de Lausanne, EPFL, Switzerland}
\email{sandro.pfammatter@epfl.ch}



\newtheorem{theorem}{Theorem}[section]
\newtheorem{lemma}[theorem]{Lemma}
\newtheorem{prop}[theorem]{Proposition}
\newtheorem{corollary}[theorem]{Corollary}

\theoremstyle{definition}
\newtheorem{definition}[theorem]{Definition}

\newtheorem{example}[theorem]{Example}
\newtheorem{remark}[theorem]{Remark}

\usepackage[
backend=biber,
style=alphabetic,
url=false,
isbn=false,
]{biblatex}
\newcommand{\N}{{\mathbb N}}

\newcommand{\Aut}{{\rm Aut}}
\newcommand{\Out}{{\rm Out}}

\newcommand{\op}{\text{op}}

\newcommand{\DM}{\mathbb{D}(\mathcal{M})}
\newcommand{\WM}{\mathbb{W}(\mathcal{M})}
\newcommand{\MM}{\mathcal{M}}
\newcommand{\mm}{\mathbb{M}}
\newcommand{\hh}{\mathbb{H}}

\renewcommand{\phi}{\varphi}

\newcommand{\catname}[1]{{\mathbf{#1}}}
\newcommand{\Set}{\catname{Set}}

\begin{document}
	\begin{abstract}
We initiate a systematic study of cohomology theories for partial groups, algebraic structures introduced by Chermak that generalize groups by allowing only partially defined products. Inspired by classical group cohomology, we develop two parallel approaches -- an algebraic theory based on Chermak's framework and a simplicial-set-based theory using local coefficient systems -- and show that they coincide. As an application, we illustrate how the extension theory of partial groups, as developed by Broto and Gonzalez, can be interpreted and computed using our cohomology theory, including explicit examples such as extensions of free partial groups, and compare these results with classical group extensions.

	\end{abstract}
	\maketitle
	
	\tableofcontents

	\section*{Introduction}

Cohomology theories play a central role in algebra and topology. For groups, cohomology encodes extensions, obstructions, and representation-theoretic information, while more broadly serving as a fundamental tool in homotopy theory and algebraic topology. It is natural to ask how such constructions extend beyond the classical setting of groups.

Partial groups were introduced by Chermak in his work on localities \cite{Chermak13}, where they provide an algebraic framework generalizing groups by relaxing the requirement that every word of elements be multiplicatively defined. This perspective arose from the study of fusion systems and the theory of $p$–local finite groups developed by Broto, Levi, and Oliver \cite{BrotoLeviOliver}. Partial groups retain enough of the structural properties of groups to allow for a rich theory, while being flexible enough to capture phenomena that do not fit into the classical group-theoretic framework.

Broto and Gonzalez \cite{BrotoGonzalez2021} developed an alternative viewpoint on partial groups using the language of simplicial sets. In this framework, they classified extensions of partial groups, which naturally led to the appearance of cohomology groups of a partial group and thus introduced the main question addressed in this paper.

The goal here is to initiate a systematic study of cohomology theories for partial groups. Just as group cohomology admits several constructions -- via resolutions, derived functors, or classifying spaces -- we develop two parallel approaches for partial groups: one algebraic, extending Chermak’s framework, and one simplicial, based on local coefficient systems. We prove that these coincide, showing that the cohomological invariants of partial groups are robust and independent of the chosen perspective. This unification not only clarifies the foundations but also opens the way to applications.

\medskip
\noindent
{\bf \cref{thmCohomology}.}
\emph{
    Let $\mm$ be a partial group acting on an abelian group $G$ via a homomorphism $\phi\colon \mm \to \operatorname{Aut}(G)$. Let $A\colon \Pi_1 \mm \to \catname{Ab}$ be the local coefficient system induced by $\phi$ and $G$. 
    Then the generalized group cohomology and the cohomology with local coefficients coincide; that is, $$H^n_\phi(\mm, G) \cong H^n(\mm; A).$$
}

\medskip

In the final section we return to the extension theory of partial groups, following the simplicial framework of Broto and Gonzalez. Using the cohomology theory introduced earlier, we describe and compute extensions in explicit cases. As a central example, we analyze extensions of free partial groups, and then compare with the classical setting of groups, recovering the cohomological classification of group extensions. This highlights both the continuity with established theory and the genuinely new phenomena unique to partial groups.
	
\medskip
	
\noindent
{\bf Acknowledgments.}
The work presented in this paper was carried out during a summer internship at the Laboratory for Topology and Neuroscience at EPFL. I would like to thank the EPFL “Summer in the Lab” and “Student Support” programs for making this opportunity possible, as well as the “Domaine de Villette” Foundation for their support of these programs. I am deeply grateful to Jérôme Scherer for his continued encouragement and guidance throughout this work. I also wish to thank Rémi Molinier and Carles Broto for reviewing the manuscript in advance and providing valuable feedback. Furthermore, I am grateful to Garry Kaltenrieder for offering useful background material through his 2024 master’s thesis.
	\section{Preliminaries}
	\label{sectionPreliminaries}
The idea behind the concept of partial groups arises from the useful but sometimes overly restrictive structure of a group. While preserving many of the essential properties of groups, partial groups relax the requirement that a product must be defined for every collection of elements. Specifically, in a partial group, not all elements can necessarily be multiplied together. We introduce this algebraic structure, along with other necessary and related notions, and discuss different approaches to its definition.

\subsection{Simplicial Sets}

We begin with a brief discussion of simplicial sets tailored to our purposes, referring the reader to \cite{GoerssJardine} for the general theory.

Let $\Delta$ be the category whose objects are ordered sets $$[n]=\{0 <1<\ldots<n\},\, n\geq 0$$ and morphisms are nondecreasing maps $f\colon [n]\to [m]$. The morphisms are generated by following two elementary kinds of maps: $d^i\colon [n-1]\to[n],\,0\leq i\leq n$, defined by $d^i(j) = j$ if $j<i$ and $d^i(j) = j + 1$ if $j \geq i$, and $s^i\colon [n+1] \to [n]$ defined by $s^i(j) = j$ if $j \leq i$ and $s^i(j) = j-1$ if $j>i$.

\begin{definition}
    A \emph{simplicial set} is a functor $X\colon \Delta^\op \to \Set$. Equivalently, a simplicial set is a sequence of sets $\big\{X_n=X([n])\big\}_{n\geq 0}$, together with structural maps $$d_i=X(d^i)\colon X_n\to X_{n-1}, \quad \text{ and }\quad s_i=X(s^i)\colon X_n\to X_{n+1}$$ called \emph{face maps} and \emph{degeneracies} respectively, satisfying the standard \emph{simplicial identities}.
    The elements of $X_n$ are called the \emph{$n$-simplices} of $X$. Maps between simplicial sets are natural transformations $f\colon X\to Y$, in other words, collections of maps $f_n\colon X_n\to Y_n$ that commute with face maps and degeneracies.
\end{definition}

We introduce some important examples of simplicial sets.
\begin{example}
    For $n \geq 0$, the \emph{standard $n$-simplex} $\Delta[n]$ is the simplicial set defined by $\hom_\Delta(-,[n])\colon \Delta^\op \to \Set$. It is generated by a distinguished $n$-simplex $\iota_n$, corresponding to $\mathrm{id}_{[n]}$, with all other simplices obtained from $\iota_n$ via face and degeneracy maps. 

    Moreover, for any simplicial set $X$ and $x \in X_n$, the simplex $x$ can be identified with the simplicial map $x\colon \Delta[n] \to X$ sending $\iota_n$ to $x$.
\end{example}

\begin{example}
    Let $\catname{C}$ be a small category. We define its nerve $\operatorname{Ner}(\catname{C})$ as the simplicial set with $n$-simplices the set of sequences of composable morphisms of length $n$. The face maps are given by
    $$d_i\big(c_0\xrightarrow{f_1}c_1\xrightarrow{f_2}\cdots\xrightarrow{f_n}c_n\big)=\begin{cases}
        \big(c_1\xrightarrow{f_2}c_2\xrightarrow{f_3}\cdots\xrightarrow{f_n}c_n\big) &i=0\\ \big(c_0\xrightarrow{f_1}\cdots c_{i-1}\xrightarrow{f_{i+1}f_i}c_{i+1}\cdots \xrightarrow{f_n}c_n\big) &1\leq i\leq n-1\\ \big(c_0\xrightarrow{f_1}c_1\xrightarrow{f_2}\cdots\xrightarrow{f_{n-1}}c_{n-1}\big) &i=n.
    \end{cases}$$
    The degeneracies are given by $$s_i\big(c_0\xrightarrow{f_1}c_1\xrightarrow{f_2}\cdots\xrightarrow{f_n}c_n\big)=\big(c_0\xrightarrow{f_1}c_1\xrightarrow{f_2}\cdots c_i\xrightarrow{\operatorname{Id}}c_i\cdots\xrightarrow{f_n}c_n\big).$$
\end{example}

\begin{example}\label{exCartesianProductSimplicalSets}
    Let $X$ and $Y$ be simplicial sets. The cartesian product of $X$ and $Y$, denoted by $X\times Y$, is the simplicial set with $n$-simplices $(X\times Y)_n=X_n\times Y_n$. Face and degeneracy maps are defined componentwise.
\end{example}

We now formalize the notion of homotopy in the simplicial setting.

\begin{definition}
    Let $X$ and $Y$ be simplicial sets, and let $f,g\colon Y\to X$ be simplicial maps. A \emph{homotopy from $g$ to $f$} is a map $$F\colon Y\times \Delta[1]\to X$$ such that $f=F\circ (\operatorname{Id}\times v_0)$ and $g=F\circ (\operatorname{Id}\times v_1)$, where $v_0=d_1(\iota_1)$ and $v_1=d_0(\iota_1)$ are the two vertices of $\Delta[1]$.
\end{definition}

The following operators, introduced in \cite{BrotoGonzalez2021}, are of special interest to us.

\begin{definition}
    Let $X$ be a simplicial set. For each $1\leq i \leq n$, let $e_i^n\colon X\to X_1$ be the \emph{$i$-th edge map} induced by the nondecreasing function $[1]\to [n]$ with image $i-1, i$.
    \begin{itemize}
        \item The \emph{spine operator} $\boldsymbol{e}^n$, $n\geq 1$, is defined as $$\boldsymbol{e}^n=(e_1^n,e_2^n,\ldots,e_n^n)\colon X_n\to X_1\times \cdots \times X_1.$$
        \item The \emph{back edge} or \emph{product operator} $\Pi^n$, $n\geq 1$, is the map $$\Pi^n\colon X_n\to X_1,$$ induced by the function $[1]\to [n]$ with image $\{0,n\}$. Equivalently, $\Pi^1=\text{Id}$ and $\Pi^n=d_1\circ d_1\circ\cdots\circ d_1$, $n-1$ times. We simply write $\Pi=\Pi^n$ when there is no possible confusion about the degree $n$.
    \end{itemize}
\end{definition}

We now establish the notion of opposite simplicial sets by reordering the structural maps.

\begin{definition}
Let $X$ be a simplicial set. The \emph{opposite simplicial set} $X^\op$ is defined by setting $X^\op_n \coloneqq X_n$ for all $n \geq 0$, with face maps $$d_i^\op \coloneqq d_{n-i} \colon X_n \to X_{n-1}$$ and degeneracy maps $$s_i^\op \coloneqq s_{n-i} \colon X_n \to X_{n+1}$$ for each $0 \leq i \leq n$. This reindexing reverses the simplicial structure while preserving all simplicial identities.
\end{definition}

\subsection{Partial Groups}
We now introduce the concept of a partial group, first following Chermak’s approach \cite{Chermak13}, and then turning to the simplicial perspective of Broto and Gonzalez \cite{BrotoGonzalez2021}.

\subsubsection{Chermak's Partial Groups}

Let $\MM$ be a nonempty set. We consider the free monoid $\WM$ on the set $\MM$. An element in $\WM$ is a finite sequence of elements of $\MM$. We represent a word $u\in \mathbb{W}(\MM)$ formed by the elements $x_1,\ldots,x_n\in\MM$ by the symbol $$u=(x_1,\ldots,x_n).$$ Given words $u=(x_1,\ldots,x_n)$ and $v=(y_1,\ldots,y_m)$ in $\mathbb{W}(\MM)$ the concatenation will be abbreviated by $$u\circ v =(x_1,\ldots,x_n,y_1,\ldots,y_m).$$
We further associate to $\mathbb{W}(\MM)$ a length function $l\colon\mathbb{W}(\MM)\to \N$ which maps each element in $\mathbb{W}(\MM)$ to the length of the sequence of element of $\MM$ which forms it. There is a unique word of length $0$, namely the empty word $(\emptyset)$. We naturally identify the set $\MM$ with the words of length $1$ in $\mathbb{W}(\MM)$.

\begin{definition}\label{dfnPartialMonoid}
    Let $\MM$ be a set, and let $\DM\subseteq\mathbb{W}(\MM)$ be a subset such that
    \begin{enumerate}[label = (\roman*)]
        \item $\MM\subseteq \DM$;
        \item if $u\circ v\in \DM$ then $u,v\in \DM$.
    \end{enumerate}
    A mapping $\Pi\colon\DM\to\MM$ is a \emph{product} if
    \begin{enumerate}[label=(P\arabic*)]
        \item\label{axiomP1} $\Pi$ restricts to the identity on $\MM$;
        \item\label{axiomP2} if $u\circ v\circ w\in \DM$ then $u\circ \Pi(v)\circ w\in \DM$ and $$\Pi(u\circ v \circ w)=\Pi\big(u\circ \Pi(v)\circ w\big).$$
    \end{enumerate}
    The \emph{unit} of a product $\Pi$ is defined as $1\coloneqq\Pi(\emptyset)$. A \emph{partial monoid} is a triple $\big(\MM,\DM,\Pi\big)$, where $\Pi$ is a product defined on $\DM$.
\end{definition}

A monoid $(M,*,1)$ can be regarded as a partial monoid by taking the domain of definition to be the set of all words, $\mathbb{D}(M)=\mathbb{W}(M)$, with the product map given by the usual multiplication.

\begin{definition}\label{dfnPartialGroupChermark}
    Let $\big(\MM,\DM,\Pi\big)$ be a partial monoid.
    An \emph{inversion} on $\MM$ is an involutory bijection $x\mapsto x^{-1}$ on $\MM$ together with the associated mapping $u\mapsto u^{-1}$ on $\WM$ given by $$u=(x_1,\ldots,x_n)\mapsto (x_n^{-1},\ldots,x_1^{-1})=u^{-1}.$$
    A partial group is a partial monoid $\big(\MM,\DM,\Pi\big)$ together with an inversion such that
    \begin{enumerate}[label=(I\arabic*)]
        \item\label{axiomI1} if $u\in \DM$ then $u^{-1}\circ u\in \DM$ and $\Pi(u^{-1}\circ u)=1$.
    \end{enumerate}
\end{definition}
To simplify the notation, we often use $\MM$ to refer to a partial group $\big(\MM,\DM,\Pi,(-)^{-1}\big)$. Partial groups enjoy group-like properties (associativity on the domain of definition, units, inverses, cancellation), see \cite[Lemma 2.2]{Chermak13} for details.

Furthermore, we note that if $\MM$ is a partial group with $\DM = \WM$, then $\MM$ admits the structure of a group under the binary operation $$(u,v)\mapsto \Pi(u\circ v).$$ Conversely, any group $G$ can be viewed as a partial group by setting $\mathbb{D}(G) = \mathbb{W}(G)$ and defining the product and inversion maps via those of $G$.

Just as with groups, there is a natural notion of a homomorphism between partial groups that preserves their algebraic structure.

\begin{definition}
    Let $\big(\MM,\DM,\Pi,(-)^{-1}\big)$ and $\big(\MM',\mathbb{D}(\MM'),\Pi',(-)^{-1}\big)$ be partial groups. A mapping $\beta\colon \MM\to \MM'$ is called a \emph{morphism of partial groups} if
    \begin{enumerate}[label=(\roman*)]
        \item $\beta^*\big(\DM\big)\subseteq \mathbb{D}(\MM')$; and
        \item $\beta\big(\Pi(u)\big)=\Pi'\big(\beta^*(u)\big)$ for all $u\in \DM$,
    \end{enumerate}
    where $\beta^*\colon\WM\to \mathbb{W}(\MM')$ is the map induced by $\beta$. The homomorphism $\beta$ is an \emph{isomorphism} if there is a homomorphism of partial groups $\beta'\colon \MM'\to \MM$ such that $\beta\circ\beta'=\text{Id}_{\MM}$ and $\beta'\circ \beta=\text{Id}_{\MM'}$.
\end{definition}

Morphisms of partial groups satisfy properties analogous to those of group homomorphisms: they preserve both units and inverses, see \cite[Lemma 3.2]{Chermak13}.

Partial groups with their homomorphisms constitute a category $\catname{Part}$, with composition and identities defined in the natural way. The category of groups forms a full subcategory of $\catname{Part}$.

We now examine a fundamental example in the theory of partial groups: the free partial group construction over a pointed set. This construction was introduced and studied in detail by Salati \cite[Section 1.1]{Salati2023}.

\begin{example}\label{exFreePartialGroups}
    Let $(X,1)$ be a pointed set and define $Y \coloneqq \{1\} \sqcup X^* \sqcup \widetilde{X}^*$, where $X^* = X \setminus \{1\}$ and $\widetilde{X}^*$ denotes a distinct copy of $X^*$. We equip $Y$ with an involutory bijection $i$ defined by $i(1) = 1$ and $i(x) = \tilde{x}$ for $x \in X^*$. 
    
    The domain $\mathbb{D}(Y) \subseteq \mathbb{W}(Y)$ consists of words formed by alternating finite strings $$(\ldots,x,\tilde{x},x,\tilde{x},\ldots)$$ for some $x \in X^*$, with any finite number of $1$'s inserted at any position. We say such words are built on $x \in X^*$. The product map $\Pi \colon \mathbb{D}(Y) \to Y$ is defined for $u \in \mathbb{D}(Y)$ built on $x \in X^*$ by $$\Pi(u)=\begin{cases}
            x &\text{if the number of $x$ is greater than that of $\tilde{x}$}\\
            1 &\text{if the number of $x$ equals that of $\tilde{x}$}\\ \tilde{x} & \text{if the number of $x$ is lower than that of $\tilde{x}$.}
        \end{cases}$$
    This construction yields a partial group $\mathbb{X} = \big(Y, \mathbb{D}(Y), \Pi, i\big)$.

    We can extend to a functor which is left adjoint to the forgetful functor $U \colon \catname{Part} \to \Set_*$, defined by $U(\MM) = (\MM, 1_\MM)$.

    We may also consider the precomposition with the functor $\Set \to \Set_*$ that assigns to a set $X$ the pointed set $\big(X \sqcup \{1\},1\big)$. The resulting functor is left adjoint to the forgetful functor $U \colon \catname{Part} \to \Set$ given by $U(\MM) = \MM$.
\end{example}

\subsubsection{Simplicial Description of Partial Groups}
We now introduce the concept of a partial group from the simplicial perspective following Broto and Gonzalez in \cite{BrotoGonzalez2021} and explain its equivalence with Chermak's earlier definition.

\begin{definition}\label{dfnPartialMonoidBroto}
    A \emph{partial monoid} is a nonempty simplicial set $\mm$ satisfying
    \begin{enumerate}[label=(PM\arabic*)]
        \item $\mm$ is reduced, i.e., $\mm_0$ consists of a unique vertex;
        \item the spine operator $\boldsymbol{e}^n\colon \mm_n\to (\mm_1)^n$ is injective for all $n\geq 1$.
    \end{enumerate}
\end{definition}

We easily see that this definition is equivalent to Definition~\ref{dfnPartialGroupChermark}. Indeed using Chermak's notation, we can regard the partial monoid $\mm$ as having underlying set $\MM=\mm_1$, the set of edges of $\mm$, and the domain of multiplication is given by $$\mathbb{D}=\big\{(\emptyset)\big\}\cup \bigsqcup_{n\geq 1}\big\{(x_1,\ldots,x_n)\in (\mm_1)^n \, \vert \, \exists\sigma\in \mm_n, \boldsymbol{e}^n(\sigma)=(x_1,\ldots,x_n)\big\}.$$
The multiplication $\Pi$ corresponds to the product operator of $\mm$, and the unit is $1=s_0(*)$, where $*$ is the unique vertex of $\mm$. Conversely, given a partial monoid $\MM$ defined as in Definition~\ref{dfnPartialGroupChermark}. We define a simplicial set $\mm$ with $\mm_0=\{*\}$ and for $n\geq 1$, $$\mm_n=\big\{u\in \DM\,\vert\,l(u)=n\big\}.$$ We define the face maps $d_i\colon\mm_n\to\mm_{n-1}$ by
$$d_i\big((x_1,\ldots,x_n)\big)=\begin{cases} \big(x_2,\ldots,x_n\big) &i=0\\\big(x_1,\ldots,\Pi(x_i,x_{i+1}),\ldots,x_n\big) &1\leq i\leq n-1 \\ \big(x_1,\ldots,x_{n-1}\big) &i=n\end{cases}$$
and the degeneracies $s_i\colon \mm_n\to \mm_{n+1}$ by $$s_i\big((x_1,\ldots,x_n)\big)=(x_1,\ldots,x_i,1,x_{i+1},\ldots,x_n).$$
The axiom~\ref{axiomP2} ensures that both the face maps $d_i$ and degeneracy maps $s_i$ are well-defined, guaranteeing that
$$d_i\big((x_1,\ldots,x_n)\big) \in \mm_{n-1} \quad \text{and} \quad s_i\big((x_1,\ldots,x_n)\big) \in \mm_{n+1}$$
for all valid indices $i$ and tuples $(x_1,\ldots,x_n) \in \mathbb{M}_n$.
One easily verifies that this yields a reduced simplicial set with injective spine operator, i.e., a partial monoid as in Definition~\ref{dfnPartialMonoidBroto}

The above correspondence suggests the following notational conventions introduced in \cite{BrotoGonzalez2021}. For a partial monoid $\mm$, a simplex $x\in \mm_n$ with spine $\boldsymbol{e}^n(x)=(x_1,\ldots,x_n)$ is written $$x=[x_1|x_2|\ldots|x_n].$$ Concerning the multiplication, we write $1=s_0(*)$, where $\mm_0=\{*\}$, and $$\Pi[x_1|\ldots|x_n]=x_1\cdot x_2\cdot \ldots \cdot x_n\in \mm_1.$$

\begin{definition}
    An \emph{inversion} in a partial monoid $\mm$ is an anti-involution $\nu\colon \mm\to\mm$, i.e., a simplicial map $\nu\colon\mm\to \mm^\op$ with inverse $\nu^{-1}=\nu^\op$, satisfying
    \begin{enumerate}[label=(I\arabic*)]
        \item if $u\in \mm_n$ with $n\geq1$ then $[\nu(u)|u]\in \mm_{2n}$ and $\Pi[\nu(u)|u]=1$.
    \end{enumerate}
    A \emph{partial group} is a partial monoid together with an inversion.
\end{definition}

Continuing with our notational conventions, we won't usally specify the map $\nu$ and instead, we write $u^{-1}=\nu(u)$. The simplicial identities imply that $[u_1|\ldots|u_n]^{-1}=[u_n^{-1}|\ldots|u_1^{-1}]$.

The equivalence between different notions of partial monoids established above naturally extends to partial groups.

\subsubsection{Reflective Subcategory of Groups Inside Partial Groups}
As previously mentioned, Chermak's definition allows us to interpret any group $G$ as a partial group in a canonical way. From the simplicial perspective, this corresponds to taking the nerve of the one-object groupoid $\mathcal{B}G$ associated to $G$. This nerve construction $\operatorname{Ner}(\mathcal{B}G)$ naturally inherits a partial group structure from the group operations of $G$. This construction yields what is commonly called the \emph{bar resolution} of the group $G$ and denoted by $BG$.

Lemoine and Molinier proved in \cite{LemoineMolinier} that the full subcategory of groups in $\catname{Part}$ is reflective. The corresponding localization functor is the fundamental group functor, obtained by viewing partial groups as a full subcategory of the category of simplicial sets. A natural question arises: with respect to which set of morphisms is this localization induced? 

We denote by $F_P(n)$ the free partial group on $n$ generators and by $F_G(n)$ the free group on $n$ generators. There are natural inclusion morphisms $\iota_n \colon F_P(n) \to F_G(n)$ induced by the identity on the generators. Let $S = \big\{\iota_n \colon F_P(n) \to F_G(n) \mid n \in \mathbb{N}\big\}.$
We can show that a partial group is $S$-local if and only if it is a group. This follows directly from the observation that partial group homomorphisms from $F_P(n)$ to a partial group $\MM$ are determined by $n$-tuples of elements in $\MM$, whereas homomorphisms from $F_G(n)$ to $\MM$ correspond to $n$-tuples whose free monoid lies within the domain of definition of the product on $\MM$.
    
We further observe that it is not sufficient to fix $\iota_n$ for a single $n \in \mathbb{N}$; that is, there exists a partial group which is $\iota_n$-local but not a group. For example, consider the partial group defined on $\MM = F_G(n+1)$ with domain $$\mathbb{D}(\MM) = \bigcup_{H \in \mathcal{H}} \mathbb{W}(H),$$ where $\mathcal{H}$ is the set of subgroups of $F_G(n+1)$ generated by at most $n$ elements. Using the product and inversion from $F_G(n+1)$, this defines a partial group. It is $\iota_n$-local, since for every $n$-tuple in $F_G(n+1)$, the free monoid on this $n$-tuple lies in the domain of definition. However, $\MM$ is not a group, because the $(n+1)$-tuple of generators $(a_1,\ldots,a_{n+1})$ does not lie in $\mathbb{D}(\MM)$.

\subsubsection{Homotopies and Normalizers}

We now examine homomorphisms and their homotopies in the category of partial groups, which naturally lead to the notions of automorphism groups and outer automorphism groups. Our exposition follows the treatment of Broto and Gonzalez \cite{BrotoGonzalez2021}.

The behavior of homotopies between partial group homomorphisms is especially tractable, as demonstrated by the following lemma.

\begin{lemma}
    Let $f,g\colon \hh\to\mm$ be two homomorphisms of partial monoids. Then, the following holds.
    \begin{enumerate}[label=(\roman*)]
        \item A homotopy $F\colon\hh\times\Delta[1] \to \mm$ from $g$ to $f$ is uniquely determined by the $1$-simplex $\eta\colon F(1,\iota_1)\in \mm_1$, which satisfies $\eta\cdot g(x)=f(x)\cdot\eta$, for all $x\in \hh_1$.
        \item A simplex $\eta\in \mm_1$ determines a homotopy from $g$ to $f$ if and only if, for each simplex $[x_1|\ldots|x_n]\in\hh$, the following conditions are satisfied:
        \begin{enumerate}[label= (\alph*)]
            \item $\omega_k=[f(x_1)|\ldots|f(x_k)|\eta|g(x_{k+1})|\ldots|g(x_n)]\in\mm$ for each $k=0,\ldots,n$; and
            \item $\Pi(\omega_0)=\Pi(\omega_1)=\ldots=\Pi(\omega_n)$
        \end{enumerate}
    \end{enumerate}
    
    In addition, if $\mm$ and $\hh$ are partial groups, then the following holds.
    \begin{enumerate}[label=(\roman*)]
        \setcounter{enumi}{2}
        \item A homotopy $F\colon\hh \times \Delta[1]\to\mm$ is uniquely determined by either one of the functions, $F_0=F|_{\hh\times \{v_0\}}$ or $F_1=F|_{\hh\times\{v_1\}}$, together with $\eta=F(1,\iota_1)\in\mm_1$.
    \end{enumerate}
\end{lemma}
\begin{proof}
    A proof of this result can be found in \cite[Lemma 2.5]{BrotoGonzalez2021}.
\end{proof}

For two partial group homomorphisms $f,g\colon \hh \to \mm$, we write $$f \xleftarrow{\eta} g$$ to denote a homotopy from $g$ to $f$ determined by an element $\eta \in \mm_1$, as established in the preceding lemma. Furthermore, since $\eta$ and $f$ determine $g$, we write sometimes $g=f^\eta$. Likewise, $f={}^\eta g$. 

Given an element $\eta \in \mm_1$ inducing a homotopy from $\operatorname{Id}_\mm$ to some endomorphism, we adopt the notation $c_\eta \coloneqq {}^\eta\operatorname{Id}_\mm$.

\begin{definition}
    The \emph{normalizer} $N(\mm)$ of a partial group $\mm$ is the set of elements $\eta\in \mm_1$ that define a homotopy of the identity on $\mm$:
    $$N(\mm)=\{\eta\in\mm_1\mid \exists c_\eta \xleftarrow{\eta} \operatorname{Id}_\mm \}.$$
    The \emph{center} of $\mm$ is the subset $Z(\mm)\subseteq N(\mm)$ of elements $\eta$ defining a self-homotopy of the identity, i.e., $c_\eta=\operatorname{Id}_\mm$.
\end{definition}

For any $\eta\in N(\mm)$, the map $c_\eta$ can be seen as a standard conjugation map. Indeed, $c_\eta$ is a well-defined automorphism, characterized by the formula $c_\eta(x)=\eta\cdot x \cdot \eta^{-1}$, for each $x\in \mm_1$, as shown below. It turns out that $N(\mm)$ is a group and $c$ defines a group homomorphism $c\colon N(\mm)\to \operatorname{Aut}(\mm)$. By analogy with the case of groups, the automorphisms in the image of $c$ are called inner, and the cokernel is defined to be the outer automorphisms. We formalize this idea in the following fundamental exact sequence.

\begin{prop}\label{propExactSequenceAutomorphismGroups}
    Let $\mm$ be a partial group. Then, $N(\mm)$ is a subgroup of $\mm$, $Z(\mm)$ is an abelian subgroup of $N(\mm)$, and there is an exact sequence $$1\to Z(\mm)\to N(\mm) \xrightarrow{c} \operatorname{Aut}(\mm)\to \operatorname{Out}(\mm)\to 1,$$ where $c\colon N(\mm)\to \operatorname{Aut}(\mm)$ is the map that sends $\eta\in N(\mm)$ to $c_\eta\in\operatorname{Aut}(\mm)$ and $\operatorname{Out}(\mm)$ is the set of homotopy classes of automorphisms of $\mm$. In particular, $N(\mm)$ and $Z(\mm)$ are invariant by automorphisms of $\mm$, and $\operatorname{Out}(\mm)$ is a group.
\end{prop}
\begin{proof}
    A proof of this result can be found in \cite[Proposition 2.11]{BrotoGonzalez2021}.
\end{proof}

\begin{remark}
    If $G$ is an ordinary group, viewed as a partial group via the bar construction, then by \cite[Lemma 2.9]{BrotoGonzalez2021} we have $$Z(BG) = Z(G), \quad N(BG) = G,$$ where $Z(-)$ and $N(-)$ denote the center and normalizer of a partial group. Likewise, $\Aut(BG) \cong \Aut(G)$ and $\Out(BG) \cong \Out(G)$. Hence the partial group constructions recover the classical ones in this case.
\end{remark}

	\section{Cohomology Theories}
	\label{sectionCohomology}

In this section, we present several approaches to cohomology for partial groups. First, we develop a theory inspired by classical group cohomology. Subsequently, we explore an alternative approach based on the underlying simplicial set structure of partial groups.

\subsection{Generalization of Group Cohomology}\label{secGGroupCohomology}
We first describe a cohomology theory analogous to group cohomology, see \cite[Section 2]{EilenbergMacLane1947}. We conjecture that this cohomology theory coincides with the framework used by Broto and Gonzalez in their preprint \cite{BrotoGonzalez2021}, though no explicit definition appears there. We are led to this conjecture by the fact that a direct proof of the classification of extensions can be obtained using this definition, as was carried out in Garry Kaltenrieder's master's thesis.

To develop this theory, we begin by explaining how a partial group can act on an ordinary group. Let $\mm$ be a partial group and $G$ an abelian group. Since the automorphism group $\operatorname{Aut}(G)$ forms a partial group in a natural way, an action of $\mm$ on $G$ is defined to be a partial group homomorphism $\phi \colon \mm \to \operatorname{Aut}(G)$.

We now define the cohomology with respect to the action $\phi$. Consider the cochain groups
$$C^n_\phi(\mm,G)=\hom_{\Set}(\mm_n,G).$$
The coboundary map $\delta^{n-1}_\phi\colon C^{n-1}_\phi(\mm,G)\to C^{n}_\phi(\mm,G)$ is given by
$$\delta^{n-1}_\phi(\psi)\big([x_1|\ldots|x_n]\big)=\phi(x_1)\Big(\psi\big([x_2|\ldots|x_n]\big)\Big)+\sum_{i=1}^n(-1)^i\psi\big(d_i[x_1|\ldots|x_n]\big)$$
for $\psi\in C^{n-1}_\phi(\mm,G)$ and $[x_1|\ldots|x_n]\in \mm_n$. We will next show that this construction yields a cochain complex.

\begin{lemma}\label{lemmaCohomologyChainComplexGrouplike}
Let $\mm$ be a partial group and $G$ an abelian group equipped with an action $\phi\colon\mm\to\operatorname{Aut}(G)$. Then the coboundary operators $\delta^n_\phi$ satisfy $\delta^{n}_\phi\circ\delta^{n-1}_\phi = 0$ for all $n \geq 1$, making $(C^\bullet_\phi(\mm,G),\delta_\phi)$ a cochain complex.
\end{lemma}
\begin{proof}
Let $\psi \colon \mm_{n-1} \to G$ be an $n$-cochain and $x = [x_1|\ldots|x_{n+1}] \in \mm_{n+1}$. Expanding the composition of coboundary operators yields
\begin{align*}
(\delta^n_\phi \circ \delta^{n-1}_\phi)(\psi)(x) &= \phi(x_1)\big(\delta^{n-1}_\phi(\psi)(d_0x)\big) + \sum_{j=1}^{n+1}(-1)^j \delta^{n-1}_\phi(\psi)(d_j x) \\
&= \phi(x_1)\left(\phi(x_2)\big(\psi(d_0d_0x)\big) + \sum_{i=1}^n (-1)^i \psi(d_i d_0 x)\right) \\
&\quad + \sum_{j=1}^{n+1} (-1)^j \left(\phi(e_1^n d_j x)\big(\psi(d_0 d_j x)\big) + \sum_{i=1}^n (-1)^i \psi(d_i d_j x)\right).
\end{align*}

We begin by analyzing the double sum. Applying the simplicial identities $d_i d_j = d_{j-1} d_i$ for $i < j$, we obtain
\begin{align*}
\sum_{j=1}^{n+1} \sum_{i=1}^n (-1)^{i+j} \psi(d_i d_j x) 
&= \sum_{i<j} (-1)^{i+j} \psi(d_{j-1} d_i x) + \sum_{i \geq j} (-1)^{i+j} \psi(d_i d_j x) \\
&= \sum_{i \leq j'} (-1)^{i+j'+1} \psi(d_{j'} d_i x) + \sum_{i \geq j} (-1)^{i+j} \psi(d_i d_j x)\\
&= 0.
\end{align*}

For the remaining terms, the properties of the partial group homomorphism give
\begin{align*}
    &\phi(x_1) \circ \phi(x_2)\big(\psi(d_0 d_0 x)\big) - \phi(e_1^n d_1 x)\big(\psi(d_0 d_1 x)\big)\\ 
&= \phi(x_1 \cdot x_2)\big(\psi(d_0 d_0 x)\big) - \phi(x_1 \cdot x_2)\big(\psi(d_0 d_0 x)\big) \\
&= 0.
\end{align*}

Finally, the last terms cancel via
\begin{align*}
&\sum_{i=1}^n (-1)^i \phi(x_1)\big(\psi(d_i d_0 x)\big) + \sum_{j=2}^{n+1} (-1)^j \phi(e_1^n d_j x)\big(\psi(d_0 d_j x)\big) \\
&= \sum_{i=1}^n (-1)^i \phi(x_1)\big(\psi(d_i d_0 x)\big) + \sum_{j=2}^{n+1} (-1)^j \phi(x_1)\big(\psi(d_{j-1} d_0 x)\big)\\
&= 0.
\end{align*}
Thus, $(\delta^n_\phi \circ \delta^{n-1}_\phi)(\psi) = 0$ as required.
\end{proof}

\begin{definition}
Let $\mm$ be a partial group with an action $\phi\colon\mm\to\operatorname{Aut}(G)$ on an abelian group $G$. The \emph{$n$-th cohomology group} $H^n_\phi(\mm,G)$ is the $n$-th cohomology group of the cochain complex $(C^\bullet_\phi(\mm,G),\delta_\phi)$.
\end{definition}
In the following remark, we clarify the relationship with classical group cohomology.

\begin{remark}
Let $H$ be a group acting on an abelian group $G$ via $\phi \colon H \to \operatorname{Aut}(G)$. When viewing $H$ as a partial group $BH$ through its simplicial set structure of its bar construction, we observe that $BH_n = H^n$. Consequently, our definitions recover precisely the classical group cohomology when defined via the bar resolution, see \cite[Chapter III.1, Example~3]{Brown1994}.
\end{remark}

\subsection{Cohomology with Local Coefficients}
We now present a second approach using cohomology with local coefficients for simplicial sets, following Bullejos, Faro, and García-Muñoz's work, \cite{Bullejos2003}. First, we review some basic concepts.

Let $X$ be a simplicial set. For an arbitrary simplicial set, there exist three equivalent definitions of the fundamental groupoid $\Pi_1X$, whose categorical equivalence is established in \cite[Section 3.1]{GoerssJardine}. We adopt the construction $GP_*X$ introduced by Gabriel and Zisman in \cite[Section 2.7.1]{GabrielZisman}, defined as the free groupoid associated to the path category $P_*X$ of~$X$. 

The path category has as objects all vertices of $X$ (i.e., elements of $X_0$) and the morphisms are freely generated by the $1$-simplices of $X$, subject to the relation that for each $2$-simplex $\sigma$ of $X$, the diagram
$$\begin{tikzcd}
v_0 \arrow[rr, "d_2\sigma"] \arrow[rd, "d_1\sigma"'] &     & v_1 \arrow[ld, "d_0\sigma"] \\ &v_2& 
\end{tikzcd}$$
commutes. The fundamental groupoid is then obtained by freely adjoining inverse morphisms to all existing non-invertible morphisms in the path category.

Furthermore, the fundamental groupoid functor $\Pi_1\colon\catname{sSet}\to \catname{Grpd}$ is left adjoint to the nerve functor $\operatorname{Ner}\colon\catname{Grpd}\to \catname{sSet}$.

A \emph{system of local coefficients} for $X$ consists of a covariant functor $$A \colon \Pi_1 X \to \catname{Ab}.$$ For the simplicial set associated to a partial group $\mm$, the local coefficient system reduces to specifying a single abelian group $G$ (since the simplicial set is reduced) together with automorphisms $A(x) \colon G \to G$ for each $x \in \mm_1$. These automorphisms must satisfy the compatibility condition $$A(x_1 \cdot x_2) = A(x_2) \circ A(x_1)$$ for all $[x_1|x_2] \in \mm_2$.

We now define cohomology with local coefficients for a simplicial set $X$. 
Let $\eta_X$ be the unit of the adjunction $\Pi_1 \dashv \operatorname{Ner}$, i.e., $$\eta_X \colon X \longrightarrow \operatorname{Ner}(\Pi_1X).$$
For each $n$-cell $x\in X_n$, if $\eta_X(x)=[x_0\xrightarrow{u_1}x_1\to\cdots\to x_n]\in \operatorname{Ner}(\Pi_1X)_n$, then $$ x_0 = d_1\cdots d_n \eta_X(x) \quad \text{and} \quad u_1 = d_2\cdots d_n \eta_X(x).$$

An \emph{$n$-singular cochain with coefficients in $A$} is a function $\psi$ assigning to each $x\in X_n$ an element $\psi(x)\in A\big(d_1\cdots d_n \eta_X(x)\big)$. We call $\psi$ \emph{normalized} if $\psi(x)=0$ for all degenerate $x\in X_n$.

The collection $C^n_{\eta_X}(X,A)$ of normalized $n$-singular cochains forms an abelian group under pointwise addition. Using the existence of inverses in groupoids, we define the coboundary operator$$\delta^{n-1}\colon C^{n-1}_{\eta_X}(X,A)\to C^{n}_{\eta_X}(X,A)$$ via the alternating sum $$(\delta^{n-1}\psi)(x) = A(u_1)^{-1}\big(\psi(d_0x)\big) + \sum_{i=1}^{n}(-1)^i\psi(d_ix),$$ where $u_1 = d_2\cdots d_{n}\eta_X(x)$. A straightforward verification using the simplical identities shows $\delta^{n+1}\circ\delta^n=0$ for all $n\geq0$, making $\big(C^\bullet_{\eta_X}(X,A),\delta\big)$ a cochain complex.

\begin{definition}
    Let $X$ be a simplicial set equipped with a system of local coefficients $A\colon \Pi_1X \to \catname{Ab}$. The \emph{$n$-th cohomology group with local coefficients in $A$}, written $H^n(X;A)$, is the $n$-th cohomology group of the cochain complex $\big(C^\bullet_{\eta_X}(X,A),\delta\big)$.
\end{definition}

\subsection{Comparison of Cohomology Theories for Partial Groups}
We now compare the two cohomology theories introduced above.

For a partial group $\mm$, the unit map of the adjunction $\Pi_1\dashv \operatorname{Ner}$ satisfies $$\eta_\mm\big([x_1|\ldots|x_n]\big) = \big(*\xrightarrow{x_1}*\xrightarrow{x_2}\cdots\xrightarrow{x_n}*\big),$$ where $*$ denotes the unique element in $\mm_0$. 

When considering the generalized group cohomology from Section~\ref{secGGroupCohomology} for an abelian group $G$ with partial group homomorphism $\phi\colon\mm\to G$, the assignment $$A(*) = G \quad \text{and} \quad A(x) = \phi(x)^{-1} \quad \text{for all } x\in \mm_1$$ yields a covariant functor $A\colon \Pi_1\mm \to \catname{Ab}$, i.e., a system of local coefficients. The functoriality follows from the identity $$A(x_1\cdot x_2) = \phi(x_1\cdot x_2)^{-1} = \big(\phi(x_1)\circ\phi(x_2)\big)^{-1} = \phi(x_2)^{-1}\circ\phi(x_1)^{-1} = A(x_2)\circ A(x_1).$$

We further observe that the cochains and coboundary operations coincide with our previous definition, modulo the normalization condition in local coefficient cohomology. By generalizing the argument of Eilenberg and MacLane \cite{EilenbergMacLane1947} to this framework, the following lemma, which is the main technical ingredient to extend the classical identification to partial groups, demonstrates that these theories produce isomorphic cohomology groups.

\begin{lemma}
Let $\mm$ be a partial group acting on an abelian group $G$ via $\phi\colon\mm\to\operatorname{Aut}(G)$, and consider the cochain complex $\big(C^\bullet_\phi(\mm,G),\delta_\phi\big)$ introduced in Lemma~\ref{lemmaCohomologyChainComplexGrouplike}. Then:
\begin{enumerate}[label=(\roman*)]
    \item Every cocycle is cohomologous to a normalized cocycle.
    \item Every normalized coboundary is the coboundary of a normalized cochain.
\end{enumerate}
\end{lemma}
\begin{proof}
    Let $\psi\colon \mm_n\to G$ be an $n$-cochain. We say $\psi$ is \emph{$i$-normalized} for $0\leq i\leq n$ if $\psi\big([x_1|\ldots|x_n]\big)=0$ whenever any of the first $i$ components equals $1$. Note that:
\begin{itemize}
    \item Every cochain is $0$-normalized.
    \item A cochain is normalized precisely when it is $n$-normalized.
\end{itemize}

From the coboundary operator's definition, if $\psi$ is $i$-normalized, then $\delta^n_\phi(\psi)$ remains $i$-normalized. This holds because all but two adjacent terms vanish immediately by the $i$-normalization condition, and these remaining terms cancel each other.

We proceed by induction to show that any $n$-cochain $\psi$ with $\delta^n_\phi(\psi)$ normalized is cohomologous to some normalized $n$-cochain $\psi_n$. Construct inductively cochains $$\psi_0, \ldots, \psi_n \in C^n_\phi(\MM,G) \quad \text{and} \quad \chi_1, \ldots, \chi_n \in C^{n-1}_\phi(\MM,G),$$
with $\psi_0 = \psi$ and, for $1 \le i \le n$, $$\psi_i = \psi_{i-1} - \delta^{n-1}_\phi(\chi_i), \qquad 
\chi_i = (-1)^{i-1} \, \psi_{i-1} \circ s_{i-1}.$$
Observe that $\psi$ and $\psi_i$ are cohomologous for all $i$. 

We show by induction, that $\psi_i$ is $i$-normalized for all $i$. For the base case $i=0$, the claim is trivial. For $i=1$, we must verify $\psi_1\big([1|x_2|\ldots|x_n]\big)=0$ for all $[x_2|\ldots|x_n]\in\mm_{n-1}$. The computation shows
    \begin{align*}
        \psi_1\big([1|x_2|\ldots|x_n]\big) =& \psi\big([1|x_2|\ldots|x_n]\big)-\phi(1)\big(\chi_1\big([x_2|\ldots|x_n]\big)\\
        &+ \chi_1\big([1\cdot x_2|\ldots|x_n]\big) + \sum_{i=2}^{n-1}(-1)^{i+1}\chi_1\big([1|x_2|\ldots|x_i\cdot x_{i+1}|\ldots|x_n]\big)\\
        &+ (-1)^{n+1}\chi_1\big([1|x_2|\ldots|x_{n-1}]\big)\\
        =& \psi\big([1|x_2|\ldots|x_n]\big) -\psi\big([1|x_2|\ldots|x_n]\big)\\
        &+ \psi\big([1|x_2|\ldots|x_n]\big) + \sum_{i=2}^{n-1}(-1)^{i+1}\psi\big([1|1|x_2|\ldots|x_i\cdot x_{i+1}|\ldots|x_n]\big)\\
        &+ (-1)^{n+1}\psi\big([1|1|x_2|\ldots|x_{n-1}]\big)\\
        =&\delta^n_\phi(\psi)\big([1|1|x_2|\ldots|x_n]\big)\\
        =& 0.
    \end{align*}
    The vanishing follows because $\delta^n_\phi(\psi)$ is normalized by hypothesis.
    Suppose the statement holds for some $i \geq 1$. To prove it for $i+1$, we first observe that since $\psi_i$ is $i$-normalized, the construction $\chi_{i+1} = (-1)^{i-1}\psi_i \circ s_{i-1}$ makes $\chi_{i+1}$ $i$-normalized as well. As established earlier, this implies $\delta^{n-1}_\phi(\chi_{i+1})$ remains $i$-normalized, and consequently $\psi_{i+1} = \psi_i - \delta^{n-1}_\phi(\chi_{i+1})$ is also $i$-normalized.

    The crucial step is verifying $\psi_{i+1}(s_i x) = 0$ for all $x \in \mm_{n-1}$. Computing explicitly
    \allowdisplaybreaks[0]
    \begin{align*}
        \psi_{i+1}(s_i x) =& \psi_i\big([x_1|\ldots|x_i|1|x_{i+1}|\ldots|x_{n-1}]\big)-\delta^{n-1}_\phi(\chi_{i+1})\big([x_1|\ldots|x_i|1|x_{i+1}|\ldots|x_{n-1}]\big) \displaybreak[3] \\
        =& \psi_i\big([x_1|\ldots|x_i|1|x_{i+1}|\ldots|x_{n-1}]\big) -\phi(x_1)\big(\chi_{i+1}\big([x_2|\ldots|x_i|1|x_{i+1}|\ldots|x_{n-1}]\big)\big)\\
        &+\sum_{j=1}^{i-1} (-1)^{j-1}\chi_{i+1}\big([x_1|\ldots|x_j\cdot x_{j+1}|\ldots|x_i|1|x_{i+1}|\ldots|x_{n-1}]\big)\\
        &+(-1)^{i-1}\chi_{i+1}\big([x_1|\ldots|x_{i-1}|x_i\cdot1|x_{i+1}|\ldots|x_{n-1}]\big)\\
        &+(-1)^i\chi_{i+1}\big([x_1|\ldots|x_i|1\cdot x_{i+1}|x_{i+2}|\ldots|x_{n-1}]\big)\\
        &+ \sum_{j=i+1}^{n-2} (-1)^j \chi_{i+1}\big([x_1|\ldots|x_i|1|x_{i+1}|\ldots|x_j\cdot x_{j+1}|\ldots|x_{n-1}]\big)\\
        &+ (-1)^{n-1} \chi_{i+1}\big([x_1|\ldots|x_i|1|x_{i+1}|\ldots|x_{n-2}]\big) \displaybreak[3]\\
        \overset{(1)}{=}& \psi_i\big([x_1|\ldots|x_i|1|x_{i+1}|\ldots|x_{n-1}]\big)\\
        &+ \sum_{j=i+1}^{n-2} (-1)^{j+i} \psi_i\big([x_1|\ldots|x_i|1|1|x_{i+1}|\ldots|x_j\cdot x_{j+1}|\ldots|x_{n-1}]\big)\\
        &+ (-1)^{i+n-1} \psi_i\big([x_1|\ldots|x_i|1|1|x_{i+1}|\ldots|x_{n-2}]\big) \displaybreak[3]\\
        \overset{(2)}{=}& (-1)^i\delta^n_\phi(\psi_i)\big([x_1|\ldots|x_i|1|1|x_{i+1}|\ldots|x_{n-1}]\big)\\
        \overset{(3)}{=}& 0.
    \end{align*}
    \allowdisplaybreaks[4]
    In equality (1) we use the $i$-normalization of $\chi_{i+1}$ which causes most terms to vanish. For (2), the $i$-normalization of $\psi_i$ makes the first $i-1$ terms in the coboundary expansion zero. Finally, (3) follows because $\psi_i$ is cohomologous to $\psi$ and $\delta^n_\phi\psi$ is normalized by assumption.

    We now complete the proof. First consider the case where $\psi$ is a cocycle. Since $\delta^n_\phi(\psi) = 0$ is trivially normalized, our inductive construction yields a normalized cocycle $\psi_n$ that is cohomologous to $\psi$.
    For the second claim, suppose $\psi$ is a normalized coboundary, meaning $\psi = \delta^{n-1}_\phi(\xi)$ for some $\xi \in C^{n-1}_\phi(\mm,G)$. Applying the same inductive procedure to $\xi$ produces a normalized cochain $\xi_{n-1}$ satisfying $\delta^{n-1}_\phi(\xi_{n-1}) = \psi$.
\end{proof}
\begin{theorem}\label{thmCohomology}
    Let $\mm$ be a partial group acting on an abelian group $G$ via a homomorphism $\phi\colon \mm \to \operatorname{Aut}(G)$. Let $A\colon \Pi_1 \mm \to \catname{Ab}$ be the local coefficient system induced by $\phi$ and $G$. 
    Then the generalized group cohomology and the cohomology with local coefficients coincide; that is, $$H^n_\phi(\mm, G) \cong H^n(\mm; A).$$
\end{theorem}

This equivalence shows that both constructions capture the same underlying cohomological information for partial groups, despite arising from seemingly different perspectives. In practice, the generalized group cohomology description may be more convenient for explicit calculations, whereas the local coefficient approach offers a broader categorical framework that readily extends to other simplicial objects.

	\section{Extensions of Partial Groups}
	\label{sectionExtensions}

In this section, we develop the theory of extensions of partial groups and present several illustrative examples. Our general framework follows the approach established in \cite{BrotoGonzalez2021}. 

\subsection{Extension Theory}

We begin by recalling the notion of fiber bundles for simplicial sets, with particular emphasis on their specialized properties in the context of partial groups.

\begin{definition}
    A \emph{fibre bundle} of simplicial sets is a simplicial map $f\colon E\to B$ satisfying the following conditions:
    \begin{enumerate}[label=(\roman*)]
        \item $f$ is surjective; and
        \item there is a simplicial set $F$ such that, for any simplex $b\colon \Delta[n]\to B$, there is a isomorphism $\phi_b\colon F\times \Delta[n]\to E_b$ making the following diagram commute
        $$\begin{tikzcd}
        {F\times \Delta[n]} \arrow[rd] \arrow[r, "\phi_b","\cong"'] & E_b \arrow[d] \arrow[r]    & E \arrow[d, "f"] \\
        & {\Delta[n]} \arrow[r, "b"] & B               
        \end{tikzcd}$$
        where $E_b\to \Delta[n]$ is the pullback of $E\to B$ along $b$.
    \end{enumerate}
    The simplicial set $F$ is called the \emph{fibre} of the bundle, $B$ is the \emph{base space} and $E$ the \emph{total space}.
\end{definition}
\begin{definition}
    Let $\mm$ and $\hh$ be partial groups. An \emph{extension} of $\hh$ by $\mm$ is a fibre bundle $\tau\colon \mathbb{E}\to \hh$ with fibre $\mm$.
\end{definition}
\begin{definition}
    Two extensions $\mm\to\mathbb{E}_i\to \hh$, $i=1,2$, are \emph{equivalent} if there is an isomorphism $\psi\colon \mathbb{E}_1\to \mathbb{E}_2$ such that the diagram $$\begin{tikzcd}
    \mm \arrow[r] \arrow[d, equal] & \mathbb{E}_1 \arrow[r] \arrow[d, "\psi"] & \hh \arrow[d, equal] \\ \mm \arrow[r] &\mathbb{E}_2 \arrow[r] & \hh \end{tikzcd}$$
    is commutative.
\end{definition}
This definition is consistent with the concept of fibre homotopy equivalence of fibre bundles, since homotopy equivalences of partial groups are isomorphisms, see \cite[Lemma 2.12]{BrotoGonzalez2021}.

Given an extension of partial groups $\mm \to \mathbb{E} \xrightarrow{\tau} \hh$, we first observe that the total space $\mathbb{E}$ forms a partial monoid. The reduction property of $\mathbb{E}$ is immediate.
Now consider $n$-simplices $x,y \in \mathbb{E}$ with identical spines, $\boldsymbol{e}^n(x) = \boldsymbol{e}^n(y)$. Since $\tau$ preserves spines, we have $\tau(x) = \tau(y)$ because $\hh$ is a partial group. Setting $b \coloneqq \tau(x) = \tau(y)$, we see that $x,y \in \mathbb{E}_b$ in the fiber bundle notation. The isomorphism $\mathbb{E}_b \cong \mm \times \Delta[n]$ guarantees injectivity of the spine operator, forcing $x = y$.

The existence of an inversion operation on $\mathbb{E}$ follows directly from the structure theorem for partial group extensions, see Theorem~\ref{thmStructurePartialGroupExtensions}.
Building on the classical combinatorial description of fiber bundles via twisting functions (as appearing already in May's book \cite{May1992}), we adapt this framework, following the approach pioneered by Broto and Gonzalez.
\begin{definition}\label{dfnTwistingPair}
    Let $\mm$ and $\hh$ be partial groups. An \emph{$\hh$-twisting pair for $\mm$} is a pair of functions $(t,\eta)$, $$t\colon\hh_1\to \operatorname{Aut}(\mm)\quad \text{ and }\quad \eta\colon\hh_2\to N(\mm),$$ satisfying the following conditions:
    \begin{enumerate}[label=(\roman*)]
        \item $\eta(g,h)$ determines a homotopy $t(g)\circ t(h)\xleftarrow{\eta(g,h)}t(g\cdot h)$, for each pair $[g|h]\in\hh_2$;
        \item $t(1)=\operatorname{Id}$ and $\eta(g,1)=1=\eta(1,g)$ for all $g\in\hh_1$; and
        \item (cocycle condition) for all $[g|h|k]\in\hh_3$, $$t(g)\big(\eta(h,k)\big)\cdot\eta(g,h\cdot k)=\eta(g,h)\cdot\eta(g\cdot h,k).$$
    \end{enumerate}
\end{definition}
In particular, an $\hh$-twisting pair of $\mm$ determines an outer action of $\hh$ on $\mm$, that is, a homomorphism of partial groups $\hh\to B\operatorname{Out}(\mm)$.

\begin{theorem}\label{thmStructurePartialGroupExtensions}
    Let $\mm$ and $\hh$ be partial groups. Then, the following holds.
    \begin{enumerate}[label=(\roman*)]
        \item An $\hh$-twisting pair $\phi=(t,\eta)$ for $\mm$ defines an extension $\mm\times_\phi\hh$ of $\hh$ by $\mm$, which is a partial group with $n$-simplices $\big[(x_1,g_1)\big|\ldots\big|(x_n,g_n)\big]$ satisfying the conditions
        \begin{enumerate}
            \item $[g_1|\ldots|g_n]\in\hh_n$; and
            \item $\big[x_1\big|t(g_1)(x_2)\big|\big(t(g_1)\circ t(g_2)\big)(x_3)\ldots\big|\big(t(g_1)\circ\ldots\circ t(g_{n-1})\big)(x_n)\big]\in \mm_n$.
        \end{enumerate}
        The product and inverse formulae for $\mm\times_\phi\hh$ are, respectively, $$\Pi\big[(x,g)\big|(z,h)\big]=\big[(x\cdot t(g)(z)\cdot\eta(g,h),g\cdot h\big],$$ $$\big[(x,g)\big]^{-1}=\big[\eta(g^{-1},g)^{-1}\cdot t(g^{-1})(x^{-1}),g^{-1}\big],$$
        in terms of products and inverses in $\mm$ and $\hh$. The projection $\tau\colon \mm\times_\phi \hh\to\hh$ is defined on $n$-simplices as $\tau\big[(x_1,g_1)\big|\ldots\big|(x_n,g_n)\big]=[g_1|\ldots|g_n]$.
        \item Given an extension $\mm\to\mathbb{E}\xrightarrow{\rho}\hh$, there is an $\hh$-twisting pair $\phi=(t,\eta)$ for $\mm$, and an equivalence of extensions
        $$\begin{tikzcd}
        \mm\times_\phi\hh \arrow[d, "\tau"'] \arrow[r, "\cong"] & \mathbb{E} \arrow[d, "\rho"] \\
        \hh \arrow[r, equal] & \hh.
        \end{tikzcd}$$
        In particular, the total space of an extension of partial groups is again a partial group.
    \end{enumerate}
\end{theorem}
\begin{proof}
    A proof of this result can be found in \cite[Theorem 4.5]{BrotoGonzalez2021}.
\end{proof}

The preceding theorem establishes that an extension of partial groups $\mm \to \mathbb{E} \xrightarrow{\tau} \hh$ is classified up to equivalence by a twisting pair $\phi = (t, \eta)$. As shown in \cite{BrotoGonzalez2021}, this extension induces a canonical outer action $$\alpha \colon \hh \to B\operatorname{Out}(\mm), \quad \alpha(g) = \big[t(g)\big] \in \operatorname{Out}(\mm) \text{ for } g \in \hh_1,$$
which depends only on the equivalence class of the extension. We call $\alpha$ the \emph{outer action induced by the extension} and denote by $\mathcal{E}(\mm, \hh, \alpha)$ the set of equivalence classes of extensions of $\hh$ by $\mm$ with fixed induced outer action $\alpha$.

The classification parallels the classical theory for finite groups. We note that the center $Z(\mm)$ remains invariant under the $\operatorname{Out}(\mm)$-action on $\mm$, yielding a group homomorphism $$\epsilon \colon \operatorname{Out}(\mm) \to \operatorname{Aut}\big(Z(\mm)\big).$$
Through composition with $\alpha$, we obtain an action $\phi_\alpha$ of $\hh$ on $Z(\mm)$. The relevant cohomology groups are then those defined in Section~\ref{sectionCohomology} for this induced action.

\begin{theorem}\label{thmClassificationPartialGroupExtensions}
    Let $\mm$ and $\hh$ be partial groups, and let $\alpha\colon\hh\to B\operatorname{Out}(\mm)$ be an outer action. Then, the following holds.
    \begin{enumerate}[label=(\roman*)]
        \item The set $\mathcal{E}(\mm,\hh,\alpha)$ is nonempty if and only if a certain associated obstruction class $[\kappa]\in H^3_{\phi_\alpha}\big(\hh,Z(\mm)\big)$ vanishes.
        \item If the set $\mathcal{E}(\mm,\hh,\alpha)$ of equivalence classes of extensions of $\hh$ by $\mm$ with outer action $\alpha$ is not empty, then $H^2_{\phi_\alpha}\big(\hh,Z(\mm)\big)$ acts freely and transitively on this set. In particular, by chosing a representative of $\mathcal{E}(\hh,\mm,\alpha)$, we deduce that $$H^2_{\phi_\alpha}\big(\hh,Z(\mm)\big)\cong\mathcal{E}(\hh,\mm,\alpha).$$
    \end{enumerate}
\end{theorem}
\begin{proof}
    A proof of this result can be found in \cite[Theorem 4.11]{BrotoGonzalez2021}.
\end{proof}

\subsection{Partial Group Extensions of Groups}
We consider the special case where $\mathbb{M} = BK$ and $\mathbb{H} = BH$ are ordinary groups (viewed as partial groups via their bar constructions $BK$ and $BH$). Theorem~\ref{thmStructurePartialGroupExtensions} shows that every extension in this setting can be realized as a twisted Cartesian product $BK \times_{(t,\eta)} BH$. 
Since the domain of definition for the product operation in this case includes all possible elements (by construction), the resulting partial group must be a group, respectively the bar resolution $BG$ of some group $G$. Conversely, an extension $1\to K \to G \to H \to 1$ induces an extension of partial groups $BK\to BG\to BH$.
We aim to demonstrate that the framework developed for partial groups generalizes the known classification theory of group extensions.

First we make the above argument precise.

\begin{prop}\label{propExtensionsGroupsPartial}
    Let $H$ and $K$ be groups. The functor $B\colon\catname{Grp}\to\catname{Part}$ yields a bijective correspondance $$\{\text{extensions of $H$ by $K$}\}/\sim \; \xrightarrow[\cong]{B}\;\{\text{extensions $BK\to \mathbb{E}\to BH$}\}/\sim$$
    where $\sim$ describes the corresponding equivalence of extensions in each category.  
\end{prop}
\begin{proof}
    Using the $5$-lemma, we deduce that the equivalence in the category of groups is actually given by isomorphisms. By \cite[Lemma 2.12]{BrotoGonzalez2021}, we conclude that this also holds in the category of partial groups. Moreover, the above discussion shows that every extension $\mathbb{E}$ of $BH$ by $BK$ can be represented as the bar construction of some group $G$.

    Consider an extension 
    $$BK \xrightarrow{i} BG \xrightarrow{\tau} BH,$$ 
    i.e., a fibre bundle $\tau\colon BG \to BH$ with fibre $BK$. By definition, $\tau$ is surjective. In particular, the induced map on $1$-simplices, $\tau_1\colon G = (BG)_1 \to (BH)_1 = H$, is also surjective and a group homomorphism.

    Furthermore, by examining the unique $0$-simplex in $BH$ and the definition of the fibre bundle, we obtain an injective partial group homomorphism $i\colon BK \to BG$. Restricting to $1$-simplices yields an injective group homomorphism $i_1\colon K \to G$.
    Exactness of the sequence 
    $$1 \to K \xrightarrow{i_1} G \xrightarrow{\tau_1} H \to 1$$ 
    follows from the properties of the fibre bundle and is straightforward to verify. Hence $G$ is a group extension of $H$ by $K$. 

    It is a routine exercise to see that isomorphic fibre bundles induce equivalent group extensions, as the induced group isomorphism makes the relevant diagrams commute.

    Finally, let 
    $$1 \to K \xrightarrow{i} G \xrightarrow{\rho} H \to 1$$ 
    be a group extension. Denote by $Bi$ and $B\rho$ the induced maps on the bar constructions of $i$ and $\rho$, respectively. A simple calculation shows that $B\rho\colon BG \to BH$ is a fibre bundle with fibre $BK$, yielding an extension 
    $$BK \xrightarrow{Bi} BG \xrightarrow{B\rho} BH.$$ 
    An equivalence of group extensions induces an equivalence of partial group extensions via the isomorphism induced by the bar construction.
\end{proof}

Therefore, when discussing group extensions, we need not specify whether we consider them in the category of groups or partial groups, as these notions are equivalent. In particular, any classification result or construction carried out in one setting immediately translates to the other.

We now establish this connection between group extensions and the previously developed framework involving outer actions.
Consider therefore a group extension $$1 \to K \to G \to H \to 1.$$ Then $K \lhd G$ is a normal subgroup, as it is the kernel of a group homomorphism by hypothesis. Moreover, $G/K \cong H$ by the First Isomorphism Theorem. Since $K$ is a normal subgroup of $G$, the group $G$ acts on $K$ by conjugation. In general, this action does not descend to a conjugation action of $G/K$ on $K$. However, by considering the orbits of $K$ acting on itself by conjugation, we naturally obtain an outer action
$$\alpha \colon G/K \cong H \to \operatorname{Aut}(K)/\operatorname{Inn}(K) \cong \operatorname{Out}(K).$$

For arbitrary groups $K$ and $H$ and a group homomorphism $\alpha \colon H \to \operatorname{Out}(K)$, we denote by $\mathcal{E}(K, H, \alpha)$ the equivalence classes of extensions with quotient $H$ and kernel $K$ that induce the homomorphism $\alpha$ in the manner described above. 

Proposition~\ref{propExtensionsGroupsPartial}, together with Theorem~\ref{thmClassificationPartialGroupExtensions}, recovers the known classification theorem for group extensions, originally due to Eilenberg and MacLane~\cite{EilenbergMacLane1947_2}; see also \cite[Theorem IV.8.7 and IV.8.8]{MacLane1967} for a detailed exposition.

\begin{corollary}
    Let $K,H$ be two groups and $\alpha\colon H\to \operatorname{Out}(K)$ a group homomorphism.
    \begin{enumerate}[label=(\roman*)]
        \item The set $\mathcal{E}(K,H,\alpha)$ is nonempty if and only if a certain associated obstruction class $[\kappa]\in H^3\big(H,Z(K)\big)$ vanishes.
        \item If the set $\mathcal{E}(K,H,\alpha)$ of equivalence classes of extensions of $H$ by $K$ with outer action $\alpha$ is nonempty, then $H^2\big(H,Z(K)\big)$ acts freely and transitively on this set.
    \end{enumerate}
    The relevant cohomology is given by group cohomology of $H$ with coefficients in $Z(K)$, where the $H$-module structure is induced by $\alpha$.
\end{corollary}
\begin{proof}
    By Proposition~\ref{propExtensionsGroupsPartial}, we have a bijective correspondence:$$\{\text{extensions of $H$ by $K$}\}/\sim \; \xrightarrow[\cong]{B}\;\{\text{extensions $BK\to \mathbb{E}\to BH$}\}/\sim\,.$$
    Furthermore, a group homomorphism $\alpha\colon H \to \operatorname{Out}(K)$ corresponds uniquely to an outer action $\alpha\colon BH \to B\operatorname{Out}(BK)$. Moreover, as established in Section~\ref{secGGroupCohomology}, we have the isomorphism $$H^\bullet_{\phi_\alpha}\big(BH,Z(BK)\big)\cong H^\bullet\big(H,Z(K)\big),$$
    where the $H$-module structure on $Z(K)$ is induced by $\alpha$.
    It remains to show that the above correspondence induces a bijection between $\mathcal{E}(BK, BH, \alpha)$ and $\mathcal{E}(K, H, \alpha)$. This follows directly from \cite[Remark 4.9]{BrotoGonzalez2021}, which completes the proof.
\end{proof}

If the outer action $\alpha\colon H \to \Out(K)$ lifts to a genuine group action $\tilde{\alpha}\colon H \to \Aut(K)$ -- for instance, when $K$ is abelian -- then the first condition is automatically satisfied. In this case, there exists at least one extension, namely the semidirect product $K \rtimes_{\tilde{\alpha}} H$.

\subsection{Extensions of Free Partial Groups}

We now compute the number of equivalence classes of extensions between two free partial groups, as defined in Example~\ref{exFreePartialGroups}.

\begin{prop}\label{propExtensionsFreePartialGroups}
    Let $(X,1)$ and $(Y,1)$ be finite pointed sets with associated free partial groups $\mathbb{X}$ and $\mathbb{Y}$. Then, the number of equivalence classes of extensions of $\mathbb{X}$ by $\mathbb{Y}$ is given by
    $$\big(|Y^*|!\,2^{|Y^*|}\big)^{|X^*|},$$
    where $X^*=X\setminus \{1\}$, respectively $Y^*=Y\setminus \{1\}$.
\end{prop}
\begin{proof}
    Applying \cite[Lemma 2.9]{BrotoGonzalez2021}, we immediately deduce that normalizers (and consequently centers) of free partial groups are trivial. This implies the vanishing of all cohomology groups $H^i_{\phi_\alpha}\big(\mathbb{X},Z(\mathbb{Y})\big)$ for any outer action $\alpha\colon\mathbb{X}\to B\operatorname{Out}(\mathbb{Y})$. In particular, the obstruction class $\kappa\in H^3_{\phi_\alpha}\big(\mathbb{X},Z(\mathbb{Y})\big)$ in Theorem~\ref{thmClassificationPartialGroupExtensions} vanishes and each outer action yields exactly one equivalence class of extensions. Thus, counting equivalence classes of extensions reduces to counting outer actions. Equivalently, by following the approach of Broto and Gonzalez in \cite{BrotoGonzalez2021} for classifying extensions via the classification theorem for fibre bundles, see the proof of \cite{BrotoGonzalez2021}[Theorem 4.11], we can argue directly by applying Proposition \ref{propExactSequenceAutomorphismGroups} together with \cite{BrotoGonzalez2021}[Corollary 3.3].

    The triviality of the normalizer yields $\operatorname{Aut}(\mathbb{Y})\cong\operatorname{Out}(\mathbb{Y})$. Automorphisms are uniquely determined by their action on $\mathbb{Y}_1$. Moreover, the adjunction between free partial groups and pointed sets yields $$
    \hom_{\catname{Part}}(\mathbb{Y},\mathbb{Y}) \cong \hom_{\Set_*}\big((Y,1), (\{1\}\sqcup Y^*\sqcup \tilde{Y}^*,1)\big).$$
    Here, automorphisms correspond to permutations of $Y^*$ combined with sign choices for generators, yielding $$\operatorname{Out}(\mathbb{Y})\cong \operatorname{Aut(\mathbb{Y})} \cong \{\pm1\}^{|Y^*|} \rtimes \Sigma_{|Y^*|},$$
    where $\Sigma_{|Y^*|}$ is the symmetric group on $|Y^*|$ letters.
    
    For outer actions $\alpha\colon\mathbb{X}\to B\operatorname{Out}(\mathbb{Y})$, we again consider the bijection between the respective homomorphism sets induced by the adjunction $$\hom_{\catname{Part}}(\mathbb{X},B\operatorname{Out}(\mathbb{Y})) \cong \hom_{\Set_*}\big((X,1), (\operatorname{Out}(\mathbb{Y}),1)\big).$$
    Consequently, the number of equivalence classes of extensions is $$\big(|Y^*|!\,2^{|Y^*|}\big)^{|X^*|}.$$
\end{proof}

We now explicitly describe these extensions as follows. Given an outer action $$\alpha\colon \mathbb{X}\to B\operatorname{Out}(\mathbb{Y}) \cong B\operatorname{Aut}(\mathbb{Y}),$$ define an $\mathbb{X}$-twisting pair $(t, \eta)$ for $\mathbb{Y}$ by $$t=\alpha|_{\mathbb{X}_1}\colon \mathbb{X}_1\to \big(B\operatorname{Aut}(\mathbb{Y})\big)_1 = \operatorname{Aut}(\mathbb{Y})\quad \text{and}\quad \eta=\operatorname{const}_1\colon \mathbb{X}_2\to N(\mathbb{Y}).$$ By construction, $t$ satisfies $t(x_1 \cdot x_2) = t(x_1) \circ t(x_2)$ for all $[x_1|x_2] \in \mathbb{X}_2$, which implies the homotopy condition $$t(x_1)\circ t(x_2)=t(x_1\cdot x_2) \xleftarrow{\eta(x_1,x_2)=1} t(x_1\cdot x_2).$$ The remaining conditions of Definition~\ref{dfnTwistingPair} hold trivially.
By Theorem~\ref{thmStructurePartialGroupExtensions}, this twisting pair yields the partial group $\mathbb{Y} \times_{(t, \eta)} \mathbb{X}$, whose $n$-simplices are given by $\big[(y_1,x_1)\big|\ldots\big|(y_n,x_n)\big]$ with $[x_1|\ldots|x_n]\in \mathbb{X}_n$ and $\big[y_1\big|\alpha(x_1)(y_2)\big|\ldots\big|\alpha(x_1\cdot\ldots\cdot x_n)(y_n)\big]\in \mathbb{Y}_n.$
The product and inverse are given respectively by $$\big[(y,x)\big]\cdot \big[(\mu,\lambda)\big]=\big[(y\cdot \alpha(x)(\mu),x\cdot \lambda)\big], \quad\big[(y,x)\big]^{-1}=\big[(\alpha(x^{-1})(y^{-1}),x^{-1})\big].$$

In particular, when $\alpha$ is the trivial outer action (i.e., $\alpha(x) = \operatorname{Id}_{\mathbb{Y}}$ for all $x \in \mathbb{X}_1$), the twisted product $\mathbb{Y} \times_{(t,\eta)} \mathbb{X}$ reduces to the product $\mathbb{Y} \times \mathbb{X}$ of simplicial sets.\\

An interesting question one may consider is the nature of the difference between extensions of free partial groups by free partial groups and extensions of free groups by free groups. The case of partial groups was discussed in Proposition~\ref{propExtensionsFreePartialGroups}. Let $K = F(T)$ and $H = F(S)$ be free groups generated by finite sets $T$ and $S$ respectively. We begin by computing their cohomology groups. For any free group $F(S)$, \cite[Chapter I.4, Example 4.3]{Brown1994} provides a free resolution of the trivial $\mathbb{Z}[F(S)]$-module $\mathbb{Z}$: $$0 \to \mathbb{Z}\big[F(S)\big]^{(S)} \xrightarrow{\delta} \mathbb{Z}\big[F(S)\big] \xrightarrow{\epsilon} \mathbb{Z} \to 0,$$ where $\mathbb{Z}\big[F(S)\big]^{(S)}$ is the free $\mathbb{Z}[F(S)]$-module with basis $(e_s)_{s\in S}$, $\delta(e_s) = s-1$, and $\epsilon(g) = 1$ for all $g \in F(S)$. 
Using the description of group cohomology via $\operatorname{Ext}$-modules, we conclude that $H^n\big(F(S),A\big)$ is trivial for all $n \geq 2$ and any $F(S)$-module $A$. In particular, extensions of $F(S)$ by $F(T)$ are in correspondence with group homomorphisms $\alpha\colon F(S) \to \operatorname{Out}\big(F(T)\big)$.
    
Since the free group functor is left adjoint to the forgetful functor $U\colon \catname{Grp}\to\catname{Set}$, we deduce that $$\hom_\catname{Grp}\Big(F(S),\operatorname{Out}\big(F(T)\big)\Big)\cong \hom_\catname{Set}\Big(S,\operatorname{Out}\big(F(T)\big)\Big).$$
Hence it suffices to calculate $\operatorname{Out}\big(F(T)\big)$. When $|T| = 1$ (so that $F(T) \cong \mathbb{Z}$), the group $\mathbb{Z}$ being abelian implies that $\operatorname{Out}(\mathbb{Z}) \cong \operatorname{Aut}(\mathbb{Z}) \cong \mathbb{Z}/2\mathbb{Z}$. It follows that the number of extensions is $2^{|S|}$.
This matches the count obtained for extensions of free partial groups with a single nontrivial element.
For $|T| \geq 2$, the abelianization map $F(T) \to \mathbb{Z}^{|T|}$ induces a surjection $\operatorname{Out}\big(F(T)\big) \to \operatorname{GL}\big(|T|, \mathbb{Z}\big)$, see \cite[Chapter 1, Proposition 4.4]{LyndonSchupp2001}. As $\operatorname{GL}\big(|T|, \mathbb{Z}\big)$ is countably infinite, there are infinitely many extensions of $F(S)$ by $F(T)$ in this case.
This stands in stark contrast to the partial group case, see Proposition~\ref{propExtensionsFreePartialGroups}, where free partial groups on finite sets yield finitely many extensions. The disparity highlights fundamental differences between these categories.

\printbibliography
\end{document}